%% file: main.tex
\documentclass[11pt, a4paper]{amsart}

\input{macros.tex}

\pdfoutput=1 

\title{Contramodules for algebraic groups: the existence of mock projectives.}
\author{Dylan Johnston}
\date{}

\begin{document}

\begin{abstract} 
\noindent{}Let $G$ be an affine algebraic group over an algebraically closed field of positive characteristic. Recent work of Hardesty, Nakano, and Sobaje gives necessary and sufficient conditions for the existence of so-called mock injective $G$-modules, that is, modules which are injective upon restriction to all Frobenius kernels of $G$. In this paper, we give analogous results for contramodules, including showing that the same necessary and sufficient conditions on $G$ guarantee the existence of mock-projective contramodules. In order to do this we first develop contramodule analogs to many well-known (co)module constructions.  
\end{abstract}

\maketitle

\section*{Introduction.}
Let $G$ be an affine algebraic group defined over an algebraically closed field of characteristic $p$ which splits over a subfield $\mathbb{F}_p$. Then $G$ admits a Frobenius morphism $F: G \rightarrow G$. Let $G_r = \ker(F^r)$ denote the $r^{th}$ Frobenius kernel. In 2015, Friedlander defined a support theory for rational $G$-modules for many important classes of groups $G$ and showed that a $G$-module has trivial support if and only if it is a mock injective module, that is, a module which is injective when restricted to all Frobenius kernels. It is also shown that mock injectivity of a module is a weaker condition than injectivity, i.e., there are mock injective modules which are not injective as a $G$-module. Such modules are called proper mock injective modules. Recent work of Hardesty, Nakano and Sobaje gives an explicit description of when $G$ admits proper mock injective modules. \cite{hardesty2017existence}

In this paper, we consider the contramodule analog of the work of Hardesty, Nakano and Sobaje. That is, we aim to give a description of when $G$ admits proper mock projective contramodules, i.e., contramodules which are not projective as a $k[G]$-contramodule, but which are projective when restricted to $k[G_r]$. 

Many of the results in this paper look strikingly similar to that of the work of the aforementioned authors, suggesting that looking through the lens of contramodules may be another useful way to investigate properties of algebraic groups $G$.

The contents are as follows. In Section \ref{Sec:intro} we will give the definition of contramodules, and discuss important families of them. We will also describe the induction and restriction functors, before finishing by giving some additional constructions in the case of contramodules over a Hopf algebra.

The remaining sections are highly motivated by the work of Hardesty, Nakano and Sobaje. In Section \ref{Sec: conditions for proper mock projectives} we show that the same conditions on $G$ to ensure existence of proper mock injective modules will also ensure the existence of mock projective contramodules. 

In Section \ref{cofinite radicals} we investigate mock projective modules with certain conditions on the radical, including but not limited to finite co-dimensionally. 

\section*{Acknowledgements.}

I would like to thank my supervisor Professor Dmitriy Rumynin for the many great discussions and suggestions given to me during the completion of this work. The author acknowledges funding from the Heilbronn Institute for Mathematical Research (HIMR) and the UK Engineering and Physical Sciences Research Council under the $\lq\lq$Additional Funding Programme for Mathematical Sciences". (Grant number: EP/V521917/1).

\section{Contramodules}\label{Sec:intro}

\subsection{First definitions}

Let $C$ be a coalgebra over a field $k$. A (left) $C$-contramodule $(B,\theta_B)$, or just $B$, is a $k$-vector space $B$ equipped with a linear map $\theta_B : \Hom_k(C,B) \longrightarrow B$, called the contra-action, satisfying contra-associativity and contra-unity conditions. That is, the following two diagrams commute:

\[\hspace{-0.75ex}
 \begin{tikzcd}[scale cd=0.88]
{\Hom_k\big(C,\Hom_k(C,B)\big)} \arrow[dd, "{\otimes\, \dashv\, \Hom}"] \arrow[rr, "{\Hom_k(C,\theta_B)}"]  & & {\Hom_k(C,B)} \arrow[dd, "\theta_B"] \\
     & &       \\
    {\Hom_k(C \otimes C,B)} \arrow[r, "{\Hom_k(\Delta,B)}"]  & {\Hom_k(C,B)} \arrow[r, "{\theta_B}"] & B
\end{tikzcd}
\hspace{1.5ex}
\begin{tikzcd}[scale cd=0.88]
{\Hom_k(k,B)} \arrow[ddrr, "{\cong}"] \arrow[rr, "{\Hom_k(\epsilon,B)}"]  & & {\Hom_k(C,B)}  \arrow[dd, "{\theta_B}"] \\
     & &       \\
    & & B
\end{tikzcd}\]
where $\lq\lq \otimes \dashv \Hom"$ denotes the tensor-hom adjunction, which for any vector spaces $U,V,W$ is given by identifying $\Hom_k\big(U, \Hom_k(V,W)\big)$ and $\Hom_k\left(V \otimes_k U, W\right)$. We remark that using instead the identification $\Hom_k\big(U, \Hom_k(V,W)\big) \cong \Hom_k(U \otimes_k V, W)$ gives the definition of a right $C$-contramodule. Unless stated otherwise, contramodules will be left contramodules

Given two $C$-contramodules $B$ and $D$, let $\Hom^C(B,D)$ denote the space of contramodule homomorphisms from $B$ to $D$. That is, the linear maps $f: B \longrightarrow D$ such that the following diagram commutes: 
\[
 \begin{tikzcd}[scale cd=1]
{\Hom_k(C,B)} \arrow[dd, "{\Hom_k(C,f})"] \arrow[rr, "{\theta_B}"]  & & {B} \arrow[dd, "f"] \\
     & &       \\
    {\Hom_k(C,D)} \arrow[rr, "{\theta_D}"]  && D.
\end{tikzcd}\]

Now consider the space $\Hom_k(C,k)$, it can be given the structure of a $C$-contramodule by applying the comultiplication of the coalgebra $C$ in the first factor. That is, $\Hom_k(C,k)$ has structure map $\theta: \Hom_k\big(C,\Hom_k(C,k)\big) \longrightarrow \Hom(C,k)$ given by the composition:
\[\Hom_k(C,\Hom_k(C,k)) \cong \Hom_k(C \otimes C, k) \xrightarrow{\Hom_k(\Delta,k)} \Hom_k(C,k).\]
More generally, one may replace $k$ with any vector space $V$, obtaining what we will call the $\textit{free contramodule on }V$. One can show that there is an isomorphism of vector spaces
\[ \Hom^C\big(\Hom_k(C,V),W\big) \cong \Hom_k(V,W)\]
 for each $C$-contramodule $W$. In particular, free contramodules are projective. It follows that any contramodule $B$ is projective if and only if it is a direct summand of a free contramodule. We now see that one may construct projective contramodules by taking the dual of injective comodules. 

 \begin{lemma}\label{inj co gives proj contra}
 Let $C$ be a coalgebra, $(M,\Delta_M)$ an injective right $C$-comodule and $V$ a vector space. Then $\Hom_k(C,V)$ is a projective $C$-contramodule, with contra-action given by the composition:
 \[\Hom_k\big(C,\Hom_k(M,V)\big) \cong \Hom_k(M \otimes C, V) \xrightarrow{\Hom_k(\Delta_M,V)} \Hom_k(M,V).\]
 \end{lemma}

 \begin{proof}
As $M$ is injective, the coaction map $\Delta_M: M \longrightarrow 
M \otimes C$ splits. Applying the additive functor $\Hom_k(-,V)$ yields a split map of contramodules $\Hom_k(M \otimes C,V) \longrightarrow \Hom_k(M,V)$ with contra-actions induced from the co-action on the relevant comodules. However, as $M \otimes C$ is the cofree comodule on $M$ we have an isomorphism $\Hom_k(M \otimes C,V) \cong \Hom_k\big(C,\Hom_k(M,V))$  of contramodules, where the latter is the free contramodule on $\Hom_k(M,V)$. Thus, $\Hom_k(M,V)$ is a direct summand of a free contramodule and is therefore projective. \end{proof}

\subsection{Induction and restriction}

Let $\pi: C \longrightarrow D$ be a map of comodules. Then given a $C$-contramodule $B$ one obtains a $D$-contramodule structure on $B$ via the composition:
\[\Hom_k(D,B) \xrightarrow{\Hom_k(\pi,B)} \Hom_k(C,B) \longrightarrow B.\]
We call this the restriction to $D$ and denote it by $\textup{Res}(B)$ or $B|_D$.
Now, let $M$ be a left $D$-comodule $M$ and let $B$ be a left $D$-contramodule. Then $\text{Cohom}_D(M,B)$ denotes the cohomomorphisms between $M$ and $B$. It is a quotient vector space of $\Hom_k(M,B)$ given by the following coequaliser:
\[ 
\begin{tikzcd}[scale cd=1, sep = huge]
\text{Cohom}_D(M,B) = \text{coeq}\bigg(\Hom_k(D \otimes M,B) \ar[r,shift left=.75ex,"{\Hom_k(\Delta_{M},B)}"]
  \ar[r,shift right=.75ex,swap,"{\Hom_k(M,\theta_B)}"]
&
\Hom_k(M,B)\bigg)\end{tikzcd}\]
where $\Hom_k(M,\theta_B) : \Hom_k(D \otimes M,B)\, \cong\, \Hom_k\big(M, \Hom_k(D,B)\big)  \longrightarrow \Hom_k(M,B).$ In particular, when $M = C$ with $D$-comodule structure given by $(\pi \otimes \textup{id}) \circ \Delta_D$ we can equip $\text{Cohom}_D(C,B)$ with $C$-contramodule structure by observing it is nothing more than a quotient of the free contramodule $\Hom_k(C,B)$. This is the induction from $D$-contramodules to $C$-contramodules, we denote the resulting contramodule by $\textup{Ind}_D^C(B)$. One can show that induction and restriction form an adjoint pair. 

\begin{lemma}
$\textup{Ind}: D\textup{-Contra} \longrightarrow C\textup{-Contra}$ is left adjoint to $\textup{Res}: C\textup{-Contra} \longrightarrow D\textup{-Contra}$, that is, for all $B \in D\textup{-Contra}$ and $V \in C\textup{-Contra}$ we have
\[\Hom^C\big(\textup{Ind}_D^C(B),V\big) \cong \Hom^D\big(B,\textup{Res}_D^C(V)\big). \]
\end{lemma}

\subsection{Contramodules over a Hopf algebra}

In the case of (co-)modules over a (co-)algebra, if one in fact has a Hopf algebra structure then one may equip the relevant module category with a monoidal structure. For contramodules this is not quite the case, instead, we can produce new contramodules via a bifunctor which takes a right comodule and left contramodule as arguments. In this subsection, we explicitly describe this bifunctor and give some properties of it.

Let $(H,\nabla,\eta,\Delta,\varepsilon,S)$ be a Hopf algebra. Recall that one may write the comultipication using Sweedler's notation, that is, given $c \in H$ we write the comultiplication as $\Delta(c) = \sum c_{(1)} \otimes c_{(2)}$, with coassociativity implying that we may write $(\textup{id}\otimes\Delta)\circ \Delta(c) = (\Delta \otimes \textup{id})\circ \Delta(c) = \sum c_{(1)} \otimes c_{(2)} \otimes c_{(3)}.$
Given a right $H$-comodule $M$ and a left $H$-contramodule $B$ we may equip $\Hom_k(M,B)$ with a ``diagonal" contramodule structure via the following composition: 

\[\Hom_k\big(H,\Hom_k(M,B)\big) \xrightarrow{\Hom_k(\nabla,\Hom_k(M,B))} \Hom_k\big(H \otimes H, \Hom_k(M,B)\big) \]
\[\cong \Hom_k\big(M \otimes H, \Hom_k(H,B)\big)\xrightarrow{\Hom_k(M,\theta_B) \circ \Hom_k(\Delta_M,\Hom_k(H,B))} \Hom_k(M,B)\]
where the identification is given by $\Hom_k\big(T \otimes U, \Hom_k(V,W)\big) \cong \Hom_k\big(V \otimes T, \Hom_k(U,W)\big)$. One readily checks that this gives $\Hom_k(M,B)$ the structure of a $H$-contramodule. That is, given a Hopf algebra $H$ we have a bifunctor
\[ \Hom_k(-,-): \text{Comod-}H^{\text{op}} \times H\text{-Contra} \longrightarrow H\text{-Contra}.\]
Let $G$ be an algebraic group over a field $k$, and let $k[G]$ denote its coordinate ring. Then $k[G]$ is a Hopf algebra. Moreover, let $T \subset G$ denote the maximal torus, and $k[T]$ be its coordinate ring. Finally, let $\chi(T) \subset k[T]$ denote the weights of $T$. \cite[I.2.4]{jantzen2003representations} Recall that for a right $k[G]$-comodule $M$ the weight spaces are given by $M_\lambda = \{m \in M : \Delta_M(m) = m \otimes \lambda\}$ for $\lambda \in \chi(T).$ We may also define weight spaces for contramodules, namely, given a $k[G]$-contramodule $B$ we define the weight space with weight $\lambda \in \chi(T)$ as
\[B_\lambda = \big\{b \in B : \text{ for all } \phi \in \Hom_k\big(k[T],\langle b \rangle\big) \text{ we have } \phi(\lambda) = \theta(\phi)\big\}. \] 
We now describe the weight spaces of $\Hom_k(M,B)$. Note that we use additive notation for the weights.

\begin{lemma}
Let $M$ be a right $k[G]$-comodule and $B$ a left $k[G]$-contramodule. Then we have 
\[\Hom_k(M,B)_{\lambda + \mu} = \prod_{\alpha + \beta = \lambda + \mu} \Hom_k(M_\alpha,B_\beta). \]
\end{lemma}

\begin{proof}
Let $\alpha, \beta \in \chi(T)$ such that $\alpha + \beta = \lambda + \mu$. We calculate explicitly the image of $\Hom_k(M_\alpha,B_\beta)$ under the diagonal action. Given $h \mapsto \phi_h \in \Hom_k(H,\Hom_k(M_\alpha,B_\beta)$ we have 
\[\big(h \mapsto \phi_h\big) \longmapsto \big(h \otimes h' \mapsto \phi_{hh'} \big) \longmapsto \big(m \otimes h \mapsto (h' \mapsto \phi_{hh'}(m)\big) \longmapsto \big(m \mapsto \theta_B(f \mapsto \phi_{\alpha f}(m)\big) = \phi_{\alpha + \beta}\]
so indeed $\Hom_k(M_\alpha,B_\beta) \subset \Hom_k(M,B)_{\lambda + \mu}$. Equality follows from the fact that \[\Hom_k(M,B) = \Hom_k\Big(\bigoplus_\lambda M_\lambda, \prod_\mu B_\mu\Big) = \prod_{\lambda,\mu} \Hom_k\big(M_\lambda,B_\mu\big).\]
\end{proof}

To conclude the section, we give the following lemma, which may be thought of as a contra-analog of the tensor identity for modules. \cite[I.3.6]{jantzen2003representations} We will dub this the $\lq\lq$hom identity for contramodules".

\begin{lemma}[Hom identity]\label{Hom-identity}
Let $M$ be a right $k[G]$-comodule, and $B$ a left $k[G]$-contramodule, then 
\[ \Hom_k\left(M,\textup{Ind}_{k[H]}^{k[G]}B\right) = \textup{Ind}_{k[H]}^{k[G]}\Big(\Hom_k(M,B)\Big),\] where the $k[G]$-contramodule structure on $\Hom_k(-,-)$ is the diagonal action in both cases.
\end{lemma}

\begin{proof}
For ease of notation throughout the proof, we assign the labels 
\[ L:= \Hom_k\left(M,\textup{Ind}_{k[H]}^{k[G]}B\right)\,\,\,\,R := \textup{Ind}_{k[H]}^{k[G]}\Big(\Hom_k(M,B)\Big). \]
Recalling the definition of $\text{Cohom}$ from section \ref{Sec:intro}, we see that as vector spaces both $L$ and $R$ are quotients of $\Hom_k(k[G] \otimes M,B)$. The steps of the proof will be as follows:

\begin{enumerate}
    \item Define linear maps $\alpha : \Hom_k(k[G] \otimes M,B) \longleftrightarrow \Hom_k(k[G] \otimes M,B): \beta$ with $\alpha$ and $\beta$ inverse to one another (as maps of vector spaces).
    \item Show that $\alpha$ and $\beta$ factor to give maps from $L$ to $R$. (Note that this is a well-definedness check.)
    \item Show that in fact $\alpha,\beta$ are $k[G]$-contramodule homomorphisms.
\end{enumerate}
To begin, we define $\alpha$ and $\beta$ as follows:
\begin{align*}
    \alpha: \Hom_k(k[G] \otimes M,B) &\longrightarrow \Hom_k(k[G] \otimes M,B): \beta \\
    \phi &\longmapsto \phi \circ \mu^T \circ (\text{Id}_{k[G] \otimes M} \otimes S) \circ (\text{Id}_{k[G]} \otimes \Delta_M) \\
    \psi \circ \mu^T \circ (\text{Id}_{k[G]} \otimes \Delta_M) &\longleftarrow \psi
\end{align*}
where $\mu^T : k[G] \otimes M \otimes k[G] \longrightarrow k[G] \otimes M$ denotes a certain twisted multiplication, given by $\mu^T(f \otimes m \otimes g) = gf \otimes m.$ Concretely, we have for $f \otimes m \in k[G] \otimes M$
\[ \alpha(\phi)(f \otimes m) = \phi\big(S(m_{(1)})f \otimes m_{(0)}\big) \]
\[ \beta(\psi)(f \otimes m) = \psi\big(m_{(1)}f \otimes m_{(0)}\big). \]

We now check that these maps are inverse to one another. We will check that $\beta \circ \alpha \equiv \text{Id}_{\Hom_k(k[G] \otimes M,B)}$, checking that $\alpha \circ \beta \equiv \text{Id}$ is similar. Let $\phi \in \Hom_k(k[G] \otimes M,B)$. Then we have for any $f \otimes m \in k[G] \otimes M$:
\begin{align*}
(\beta \circ \alpha)(\phi)(f \otimes m) &= \beta\Big(f \otimes m \longmapsto \phi\big(S(m_{(1)})f \otimes m_{(0)}\big)\Big)\\ &= \Big(f \otimes m \longmapsto \phi\big(S(m_{(1)})m_{(2)}f \otimes m_{(0)}\big)\Big) \\
&= \Big(f \otimes m \longmapsto \phi\big(\varepsilon(m_{(1)})f \otimes m_{(0)}\big)\Big) = \Big(f \otimes m \longmapsto \phi\big(f \otimes m)\Big)\\
\end{align*}

Now, as vector spaces, we have (with implicit applications of the tensor hom adjunction):

\begin{tikzcd}[scale cd=1, sep = huge]
L = \text{coeq}\bigg(\Hom_k(M \otimes k[H] \otimes k[G],B) \ar[rr,shift left=.75ex,"{\Hom_k\big(M \otimes \Delta_{k[G]},B\big)}"]
  \ar[rr,shift right=.75ex,swap,"{\Hom_k\big(M \otimes k[G],\theta_B\big)}"]
&& \Hom_k(M \otimes k[G],B)\bigg).\end{tikzcd}

\begin{tikzcd}[scale cd=1, sep = huge]
R = \text{coeq}\bigg(\Hom_k(M \otimes k[H] \otimes k[G],B) \ar[rr,shift left=.75ex,"{\Hom_k\big(M \otimes \Delta_{k[G]},B\big)}"]
  \ar[rr,shift right=.75ex,swap,"{\Hom_k\big(k[G],\theta_{\Hom_k(M,B)}\big)}"]
&& \Hom_k(M \otimes k[G],B)\bigg).\end{tikzcd}

We will write $f_L = \Hom_k\big(M \otimes \Delta_{k[G]},B\big)$, $g_L = \Hom_k\big(M \otimes k[G],\theta_B\big)$ for the two maps defining $L$, and also $f_R = \Hom_k\big(M \otimes \Delta_{k[G]},B\big)$ and $g_R = \Hom_k\big(k[G],\theta_{\Hom_k(M,B)}\big)$ for the two maps defining $R$, to once again ease notation slightly. 

Now, $\alpha$ composed with the natural quotient map from $\Hom_k(M \otimes k[G],B)$ to $R$ gives us a map from $\Hom_k(M \otimes k[G],B)$ to $R$ which we also denote $\alpha$. We now want to check that this $\alpha$ factors through $L$.
In other words, we wish to check that $\Image\big(\alpha \circ (f_L - g_L)\big) \subset \Image\big(f_R - g_R\big)$. 

This is equivalent to finding a linear endomorphism $T \in \End\Big(\Hom_k( k[H] \otimes k[G] \otimes M,B) \Big)$ with $\alpha \circ (f_L - g_L) = (f_R - g_R) \circ T $. One checks that $T := \Hom_k\Big(\mu^{5,1}_{4,2} \circ \big(\text{Id}^{\otimes 2} \otimes \big((\text{Id} \otimes S^{\otimes 2}) \circ \Delta^2_M\big)\big),B\Big)$ satisfies this, where both $\mu_{ij}$ and $\mu^{ij}$ denotes multiplication given by taking the element in the $i^{th}$ factor of the tensor and the $j^{th}$ factor of the tensor, multiplying them together, and letting the resulting product replace the factor taken from the $j^{th}$ position. Concretely we have, for an algebra $A$, say,
\begin{align*}    
\mu_{ij} : \underbrace{A \otimes A \otimes \dots \otimes A}_{n \text{ copies}} &\longrightarrow \underbrace{A \otimes \dots \otimes A}_{n-1 \text{ copies}} \\
a_1 \otimes \dots \otimes a_i \otimes \dots \otimes a_j \otimes \dots \otimes a_n &\longmapsto a_1 \otimes \otimes \dots a_{i-1} \otimes a_{i+1} \otimes \dots \otimes a_ia_j \otimes \dots \otimes a_n.
\end{align*}

Similarly, to show that $\beta$ gives a well defined map from $R$ to $L$, we must find a linear endomorphism $U$ such that $(f_L - g_L) \circ U = \beta \circ (f_R - g_R).$ Once again, one readily checks that $U := \Hom_k\Big(\mu^{5,2}_{4,1} \circ \big(\text{Id}^{\otimes 2} \otimes \Delta^2_M),B\Big)$ satisfies this condition.

Thus far, we have that $\alpha: L \longleftrightarrow R : \beta$ gives an isomorphism of vector spaces. To conclude we show that, in fact, $\alpha$ is a map of contramodules. It will be sufficient to check that the following diagram commutes:

\[
 \begin{tikzcd}[scale cd=1]
{\Hom_k\Big(k[G],\Hom_k\big(M,\Hom_k(k[G],B)\big)\Big)} \arrow[dd, "{\Hom_k(k[G],\alpha})"] \arrow[rr, "{\theta_{\text{diag}}}"]  
& & {\Hom_k\Big(M,\Hom_k(k[G],B)\Big)} \arrow[dd, "\alpha"] \\  
& &       \\
{\Hom_k\Big(k[G],\Hom_k\big(M,\Hom_k(k[G],B)\big)\Big)} \arrow[rr, "{\theta_{\text{free}}}"]  
&& {\Hom_k\Big(M,\Hom_k(k[G],B)\Big)}.
\end{tikzcd}\]
Here, $\theta_{\text{diag}}$ denotes the diagonal contra-action on $\Hom_k\big(M,\Hom_k(k[G],B)\big)$, and $\theta_{\text{free}}$ denotes the free contra-action on $\Hom_k\Big(M,\Hom_k(k[G],B)\Big) \cong \Hom_k\Big(k[G],\Hom_k(M,B)\Big).$

Let $\varphi \in 
\Hom_k\Big(k[G],\Hom_k\big(M,\Hom_k(k[G],B)\big)\Big)$ be denoted as the map $f \mapsto \Big(m \mapsto \big(g \mapsto b(f,m,g)\big)\Big)$
where $b(-,-,-) : k[G] \otimes M \otimes k[G] \longrightarrow B$ , travelling vertically then horizontally gives us 
\begin{align*}
    \Big(\theta_{\text{free}} (\alpha \circ \varphi)\Big)(m) &= \bigg(m \longmapsto \Big(f \longmapsto b\big(f_{(2)},m_{(0)},S(m_{(1)})f_{(1)}\big)\Big)\bigg).
\end{align*}
On the other hand, travelling horizontally then vertically gives us 
\begin{align*}
    \Big(\alpha \circ \theta_{\text{diag}} (\varphi)\Big)(m) &= \bigg(m \longmapsto \Big(f \longmapsto b\big(m_{(1)}S(m_{(2)})f_{(2)},m_{(0)},S(m_{(3)})f_{(1)}\big)\Big)\bigg)
\end{align*}
which after observing that $m_{(0)} \otimes m_{(1)}S(m_{(2)}) \otimes m_{(3)} = m_{(0)} \otimes m_{(1)} \otimes 1$ and using that $b(-,-,-)$ is tensorial gives equality.

Finally, since $\alpha: L \rightarrow R$ is a $k[G]$-contramodule isomorphism with linear inverse $\beta: R \rightarrow L$, we deduce that $\beta$ is also a $k[G]$-contramodule isomorphism and the proof is complete. 
\end{proof}

We immediately obtain the following corollary, which we give now for later use. 

\begin{corollary}
Let $P$ be a projective $k[G]$-contramodule. Then, for any right $k[G]$-comodule $M$ we have that $\Hom_k(M,P)$, with diagonal action, is a projective $k[G]$-contramodule.
\end{corollary}

\begin{proof}
Since $P$ is projective, the contra-action map $\theta : \Hom_k(k[G],P) \longrightarrow P$ splits. Now, consider the additive functor $\Hom_k(M,-): k[G]-\text{Contra} \longrightarrow k[G]-\text{Contra} $ which equips the resulting contramodule with the diagonal action. Applying this to $\theta$ above we have 
\[ \Hom_k\big(M,\Hom_k(k[G],P)\big) \xrightarrow[]{\Hom_k(M,\theta)} \Hom_k(M,P)\]
which splits. Thus, $\Hom_k(M,P)$ is a direct summand of $\Hom_k\big(M,\Hom_k(k[G],P)\big)$. However, by the previous lemma we hav $\Hom_k\big(M,\Hom_k(k[G],P)\big) \cong \Hom_k\big(k[G],\Hom_k(M,P)\big)$, the free contramodule on the vector space $\Hom_k(M,P)$. Therefore $\Hom_k(M,P)$ is a direct summand of a free contramodule and is therefore projective. 
\end{proof}

To conclude the section, we give a final contra-analog of a construction well-known for modules over a group. Let $G = N \ltimes K$. Then one has $\text{ind}_K^G M = k[G] \otimes_{k[K]} M = k[N] \otimes M$ where $K < G$ acts on $k[N]$ via conjugation. For contramodules the obvious analog holds, we have:

\begin{lemma}\label{semi-direct-induction}
Let $G = N \ltimes K$ with associated coordinate rings $k[G], k[N]$ and $k[K]$. Then for any $k[K]$-contramodule $(M,\theta_M)$ we have:
\[Ind_{k[K]}^{k[G]}(M) \cong \Hom_k(k[N],M)\]
where the contramodule structure on the right hand side is the diagonal action and the right $k[K]$-comodule structure on $k[N]$ is induced from the conjugation action of $K$ on $k[N]$.
\end{lemma}

\begin{proof}
To begin, it will serve us well to establish some notation. Let $\iota: N \to G$ denote the natural inclusion and $\iota^* : k[G] \to k[N]$ be the corresponding map of coordinate rings. Let $\Delta_R: k[G] \to k[G] \otimes k[K]$ denote the right $k[K]$-comodule structure on $k[G]$ corresponding to left multiplication $K \times G \to G$, similarly define $\Delta_L: k[G] \to k[K] \otimes k[G]$ corresponding to right multiplication of $K$ on $G$. Finally, let $\Delta_{\text{cong}}: k[N] \to k[N] \otimes k[K]$ denote the $k[K]$-comodule structure on $k[N]$ induced from conjugation $K \times N \to N; (k,n) \mapsto knk^{-1}$.

Recall that is vector spaces we have $Ind_{k[K]}^{k[G]}(M) \cong \text{Cohom}_{k[K]}(k[G],M)$, and one equips this space with contramodule structure by realising it as a quotient of the free contramodule on $M$. Now, consider the following map:
\begin{align*}\Hom_k(k[N],M) &\longrightarrow \text{Ind}_{k[K]}^{k[G]}(M)\\
\phi &\longmapsto [\phi \circ \iota^*] 
\end{align*}
where $[ \cdot ]$ denotes the equivalence class. This is clearly an isomorphism of vector spaces. So all that remains is to check that it preserves the $k[G]$-contramodule structure. We also observe that the crux of the proof lies in checking that the $k[K]$-contramodule structure (given by restriction) is preserved. Thus, consider the following diagram, which we wish to show commutes:

\[
 \begin{tikzcd}[scale cd=1]
{\Hom_k\big(k[K],\Hom_k(k[N],M)\big)} \arrow[dd] \arrow[rr]  & & {\Hom_k\big(k[K],\text{Ind}(M)\big)} \arrow[dd] \\
     & &       \\
    {\Hom_k(k[N],M)} \arrow[rr]  && \text{Ind}(M).
\end{tikzcd}\]
Let $k \mapsto (n \mapsto m(k,n)) \in \Hom_k\big(k[K],\Hom_k(k[N],M)\big)$. Travelling horizontally and then vertically gives $\Big[ g \longmapsto m\Big(\Delta_R(g)_{(1)},\,\iota^*\big(\Delta_R(g)_{(0)}\big)\Big)\Big] \in \text{Ind}(M)$. On the other hand, travelling vertically and then horizontally gives
\begin{align*}
&\bigg[g \longmapsto \theta_M\bigg(k \longmapsto m\Big(\Delta_{\text{cong}}(g)_{(1)} \cdot k,\iota^*(\Delta_{\text{cong}}(g)_{(0)})\Big)\bigg)\bigg] \\
\equiv \,\, &\bigg[ g \longmapsto m\bigg(\Delta_\text{cong}\Big(\Delta_L(g)_{(0)}\Big)_{(1)} \cdot\Delta_L(g)_{(-1)},\, \iota^*\bigg(\Delta_{\text{cong}}\Big(\Delta_L(g)_{(0)}\Big)_{(0)}\bigg)\bigg)\bigg] \in \text{Ind}(M)   
\end{align*}
Observe that is it now sufficient to show that \[\Delta_R(g) = \Delta_{\text{cong}}\Big(\Delta_L(g)_{(0)}\Big)_{(0)} \otimes \Delta_\text{cong}\Big(\Delta_L(g)_{(0)}\Big)_{(1)} \cdot\Delta_L(g)_{(-1)}\]
or more concisely that $\Delta_R \equiv (\text{id} \otimes \mu) \circ \omega \circ (\text{id} \otimes \Delta_{\text{cong}}) \circ \Delta_L$, where $\omega(x \otimes y \otimes z) = (y \otimes z \otimes x)$ and $(\text{id} \otimes \mu)(x \otimes y \otimes z) = x \otimes yz$.

However recall that $\Delta_R$ is the map of coordinate rings associated to $K \times G \to G; (k,g) \mapsto kg$ and the right hand side is the map corresponding to the composition 
\begin{align*}
K \times G \longrightarrow K \times K \times G \longrightarrow K \times G \times K \longrightarrow G \times K \to G \\
(k,g) \longmapsto (k,k,g) \longmapsto (k,g,k) \longmapsto (kgk^{-1},k) \longmapsto kg
\end{align*}
and so they are equal, as required.
\end{proof}

\section{Conditions on \texorpdfstring{$k[G]$}{k[G]} for the existence of proper mock projective contramodules}\label{Sec: conditions for proper mock projectives}

Let $G$ be an affine group scheme over an algebraically closed field $k$.Given a subgroup scheme $H$ of $G$, recall that we defined the induction functor $\text{Ind}_{k[H]}^{k[G]}(-): k[H]-\text{Contra} \longrightarrow k[G]-\text{Contra}$ as the underlying vector space $\text{Cohom}_{k[H]}(k[G],-)$ with $k[G]$-contramodule structure given by observing that $\text{Cohom}_{k[H]}(k[G],-)$ is a quotient of the free contramodule $\Hom_k(k[G],-)$. We say that $H$ is contra-exact in $G$ if $\text{Ind}_{k[H]}^{k[G]}(-)$ is an exact functor. One says that $H$ is exact in $G$ if the induction functor of modules is exact, it turns out that these are equivalent. \cite{johnston2023} Therefore we may drop the prefix $\lq\lq$ contra" and just say that $H$ is exact in $G$. As an example, all finite subgroup schemes of $G$ are exact.
Before giving our first result on conditions on $G$ for mock projective contramodules to exist we give a useful lemma, showing that restriction takes projective contramodules to projective contramodules, provided the subscheme is exact. 

\begin{lemma}\label{res proj to proj}
Let $H$ be exact in $G$. Then the restriction functor $\textup{Res}_{k[H]}^{k[G]}$ takes projective contramodules to projective contramodules.
\end{lemma}

\begin{proof}
Let $B \in k[G]$-Contra be projective. Then $B$ is a direct summand of $\Hom_k(k[G],B)$, a free contramodule. Since restriction is an additive functor, $B$ is also a direct summand of $\Hom_k(k[G],B)$ as $k[H]$-contramodules. Since $H$ is exact in $G$, $k[G]$ is an injective $k[H]$ contramodule and so $k[G]$ is a direct summand of $k[G] \otimes k[H]$.

Now, the functor $\Hom_k(-,B): \textup{Comod}-k[H] \longrightarrow k[H]-\textup{Contra}$ is additive and thus we have that $\Hom_k(k[G],B)$ is a direct summand of $\Hom_k(k[G] \otimes k[H], B)$ as $k[H]$- contramodules. Combining both direct summand inclusions we have that $B$ is the direct summand of $\Hom_k(k[G] \otimes k[H],B) \cong \Hom_k\big(k[H],\Hom_k(k[G],B)\big)$,
where the latter is the free $k[H]$-contramodule on the vector space $\Hom_k(k[G],B)$. Thus $B$ is the direct summand of a free $k[H]$-contramodule and is therefore projective.
\end{proof}

\begin{proposition}\label{ind(k) projective when?}
Let $H$ be a finite subgroup scheme in $G$ with coordinate rings $k[H]$ and $k[G]$ respectively. Then:
\begin{enumerate}[label = \alph*)]
    \item $\textup{Ind}_{k[H]}^{k[G]}k$ is a projective $k[G]$-contramodule if and only if $k$ is a projective $k[H]$- contramodule.
    \item If the Frobenius map $F: G \longrightarrow G$ restricts to an automorphism of $H$, then $\textup{Ind}_{k[H]}^{k[G]}k$ is projective over $G_r$ for all $r>0$.
\end{enumerate}
\end{proposition}

\begin{proof}
We first prove part $a)$. For the if direction, simply observe that since restriction takes epimorphisms to epimorphisms, induction takes projective contramodules to projective contramodules. To prove the only if direction we wish for some sort of $\lq\lq$generalised Frobenius reciprocity" for contramodules, we develop this via the Grothendieck spectral sequence. \cite[Section I.4.1]{jantzen2003representations} \cite[Thm XX.9.6]{lang2012algebra}.

Observe that for any $k[G]$-contramodule $V$, the adjunction between induction and restriction may be viewed as an isomorphism of functors
\[ \Hom^{k[G]}(-,V) \circ \textup{Ind}_{k[H]}^{k[G]}\,^{\text{op}}(-) \,\cong\, \Hom^{k[H]}\big(-,V|_{k[H]}\big),\]
where notably we have $\textup{Ind}_{k[H]}^{k[G]}\,^{\text{op}}: k[H]$-Contra$^{\text{op}} \longrightarrow k[G]$-Contra$^{\text{op}}$. Since contramodule categories have enough projectives, opposite contramodule categories have enough injectives. Furthermore, $\textup{Ind}_{k[H]}^{k[G]}\,^{\text{op}}$ is exact since $H$ is a finite subgroup scheme of $G$. One checks that all other requirements to apply Grothendieck's spectral sequence (specifically special case $(2)$ of \cite[Prop I.4.1]{jantzen2003representations}) are satisfied, and so we have an isomorphism

\[ \Ext_{k[G]-\text{Contra}}^n\Big(\textup{Ind}_{k[H]}^{k[G]}W,V\Big) \,\,\cong\,\, \Ext_{k[H]-\text{Contra}}^n\big(W,V\big) \]
for each $V \in k[G]$-Contra, $W \in k[H]$-Contra. In particular for $V = W = k$ and $n>0$ we have
\[ \Ext_{k[H]-\text{Contra}}^n\big(k,k\big) = 0 \text{ for all }n > 0\]
since $\textup{Ind}_{k[H]}^{k[G]}k$ is projective by assumption. Finally, since $H$ is finite, we have an equivalence of categories between $k[H]$-Contra and Mod-$H$, the category of right $H$ modules. Now, by the theory of cohomological support varieties we have that $k$ is a projective $H$-Mod, and thus a projective $k[H]$-contramodule. \cite[Theorem 5.6 (5)] {friedlander2005representation}
For part b$)$, let $I$ be an injective $k[G]$-comodule, since $H$ is exact in $G$ the restriction of $I$ to $k[H]$ is an injective $k[H]$-comodule. \cite[Prop 2.1]{cline1977induced} If the Frobenius morphism $F$ restricts to automorphism of $H$, then $I^{(r)}$ is also an injective $k[H]$-comodule for any $r >0$. Therefore $\Hom_k\big(I^{(r)},k\big)$ is a projective $k[H]$-contramodule by Lemma \ref{inj co gives proj contra}. Furthermore, by Lemma \ref{Hom-identity} we have
\[ \Hom_k\Big(I^{(r)},\textup{Ind}_{k[H]}^{k[G]}k\Big) \,\,\cong \,\, \textup{Ind}_{k[H]}^{k[G]}\Big(\Hom_k\big(I^{(r)},k\big)\Big).\]
Since induction takes projective objects to projective objects, the right hand side (and therefore the left hand side) is a projective $k[G]$-contramodule. Furthermore, since $G_r$ is exact in $G$, restriction takes projective objects to projective objects, and so the restriction of the left hand side is a projective $k[G_r]$-contramodule. But now, as a $k[G_r]$-contramodule we have
\[ \Hom_k\Big(I^{(r)},\textup{Ind}_{k[H]}^{k[G]}k\Big) \,\, \cong \,\, \Hom_k\bigg(\bigoplus_{i=1}^{\dim(I)}k,\textup{Ind}_{k[H]}^{k[G]}k\bigg) \,\, \cong \,\, \prod_{i=1}^{\dim(I)} \textup{Ind}_{k[H]}^{k[G]}k. \]
Thus, as $k[G_r]$-contramodules, $ \textup{Ind}_{k[H]}^{k[G]}k$ is a direct summand of a projective contramodule, and so is itself projective.
\end{proof}

\begin{proposition}
Let $H$ be a finite subgroup scheme of $G$ for which every simple $k[H]$- contramodule is the restriction of a $k[H]$-contramodule. Then for any right $k[G]$-comodule $M$, $\Hom_k\big(M,\textup{Ind}_{k[H]}^{k[G}k\big)$ is a projective $k[G]$-contramodule if and only if $\Hom_k(M,k)$ is a projective $k[H]$- contramodule. 
\end{proposition}

\begin{proof}
If $\Hom_k(M,k)$ is a projective $k[H]$-contramodule, then by Lemma \ref{Hom-identity} we have that $\textup{Ind}_{k[H]}^{k[G]}\big(\Hom_k(M,k)\big) \,\, \cong \,\, \Hom_k\big(M,\textup{Ind}_{k[H]}^{k[G]}k\big)$ is a projective $k[G]$-contramodule.
Conversely, using the Grothendieck spectral sequence, as seen in the proof of Lemma \ref{Hom-identity}, we have that for any $k[G]$-contramodule $B$:
\[\Ext_{k[G]-\text{Contra}}^n\Big(\textup{Ind}_{k[H]}^{k[G]}\big(\Hom_k(M,k)\big),B\Big) \,\,\cong\,\, \Ext_{k[H]-\text{Contra}}^n\big(\Hom_k(M,k),B\big). \]
Since every simple $k[H]$-contramodule comes from a $k[G]$-contramodule, we immediately conclude that $\Hom_k(M,k)$ is projective if $\textup{Ind}_{k[H]}^{k[G]}\big(\Hom_k(M,k)\big)$ is.
\end{proof}

We may now describe conditions on an algebraic group scheme for it to have proper mock projective contramodules. In order to assist with the proof, we first give the following lemma.

\begin{lemma}\cite[Appendix A]{positselski2010homological}\label{contras_over_decomp_coalgebras}
Let $C$ be a coalgebra which is the direct sum of a family of coalgebras $C_\alpha$. Then
any left contramodule $B$ over $C$ is the product of a uniquely defined family of left
contramodules $B_\alpha$ over $C_\alpha$    
\end{lemma}

In particular, if $C$ is cosemisimple, then any contramodule over $C$ is the direct product of simple contramodules. With this fact in hand, we may now state and prove our theorem.

\begin{theorem}\label{conditions for proper mocks}
Let $G$ be an affine algebraic group scheme over a field $k$ which is defined and split over a finite subfield $\mathbb{F}_q \subset k$. Then the following are equivalent:\begin{enumerate}[label=\roman*)]
    \item $k[G]$ has proper mock projective contramodules
    \item $G$ has proper mock injective modules
    \item either $G^0$ is not a torus or $G/G^0$ has order divisible by $p$.
\end{enumerate} 
\end{theorem}

\begin{proof}
The analogous result of Hardesty, Nakano, and Sobaje gives $ii) \iff iii)$. \cite[Theorem 2.2.1]{hardesty2017existence}. To show $i) \iff iii)$ we use a similar proof technique. If $p$ does not divide the order of $G(\mathbb{F}_q)$, then $G^0$ is a torus and $G/G^0$ is a finite group of order not divisible by $p$. Thus, every element of $G$ is semisimple \cite[Theorem 2]{nagata1961complete} and so $k[G]$ is cosemisimple. Therefore, by Lemma \ref{contras_over_decomp_coalgebras} every contramodule is a direct product of simple contramodules. Thus every $k[G]$-contramodule is projective since all maps of contramodules are given as products of maps between the simple constituents, and so there cannot be any proper mock projective contramodules.
On the other hand, if $p$ divides the order of $G(\mathbb{F}_q)$ then $k$ is a non-projective $k[G(\mathbb{F}_q)]$-contramodule. Thus by Proposition \ref{ind(k) projective when?}, $\textup{Ind}_{k[H]}^{k[G]}k$ is a non-projective $k[G]$-contramodule whilst being projective as a contramodule over $k[G_r]$ for all $r > 0$.
\end{proof}

We now produce a family of non-projective $k[G]$-contramodules which are projective with respect to the fixed point subgroups of powers of the Frobenius map.

\begin{proposition}
Let $G$ be an affine algebraic group defined over $\mathbb{F}_p$. Let $P$ be a projective $k[G]$-contramodule. Then $P^{(r)}$, $r>0$, is projective as a $k[G(\mathbb{F}_q)]$-contramodule, $q = p^s$ for all large enough $s$, but is not projective as a $k[G]$-contramodule.
\end{proposition}

\begin{proof}
As $G(\mathbb{F}_q)$ is finite, it is exact in $G$ and therefore $\textup{Res}_{k[G(\mathbb{F}_q)]}^{k[G]}P$ is a projective contramodule. For $r < s$, the $r^{th}$ power of the Frobenius map is an automorphism of $G(\mathbb{F}_q)$, and so $P^{(r)}$ is also projective over $k[G(\mathbb{F}_q)]$. Finally, as a $k[G_r]$-contramodule it is trivial, thus is not projective over $k[G_r]$ and therefore cannot be projective over $k[G]$.
\end{proof}

\section{Mock projectives with cofinite radicals}\label{cofinite radicals}

In this section, $G$ is a connected reductive algebraic group scheme over a field $k$. We wish to investigate mock projective contramodules which have finite head. We begin our investigation by looking at contramodules over the coordinate ring of a unipotent group.

Let $U$ be a connected unipotent group over $k$ which is defined over $\mathbb{F}_p$ and let $k[U]$ be its coordinate ring. We first want to classify all simple $k[U]$-contramodules, it turns out that, just as in the case of $k[U]$-comodules, there is only one.

\begin{lemma}
Let $U$ be a unipotent algebraic group, then $k$ is the only simple $k[U]$-contramodule.
\end{lemma}

\begin{proof}
Let $\varepsilon: k[U] \longrightarrow k$ denote the counit map. Then $k[U] = k \oplus \ker(\varepsilon)$ where $\ker(\varepsilon)$ is a conilpotent coalgebra. Let $(S,\theta)$ be a simple $k[U]$-contramodule. Let $S' = \Image\big(\theta|_{\Hom_k(\ker(\varepsilon),S)}\big)$ denote the image of $\theta$ under the restriction to $\Hom_k\big(\ker(\varepsilon),S\big)$. Then we have $S' \subsetneq S$. \cite[Appendix A.2, Lemma 1]{positselski2010homological} Now consider the restriction of $\theta$ to $\Hom_k(k,S') \subset \Hom_k(k[U],S)$. We have by counity that $\theta|_{\Hom_k(k,S')}: (1 \mapsto s') \longmapsto s'$ and so $S'$ is a $k[U]$-subcontramodule of $S$ properly contained in $S$ and therefore $S' =0$. Thus, $S$ must be simple as a $k$-contramodule, which implies $S = k$.
\end{proof}

We now produce a proper mock projective contramodule over a unipotent group with cofinite radical.

\begin{proposition}
Let $r \geq 1$ and $q = p^r$. Then the proper mock projective $k[U]$-contramodule $B_r = \textup{Ind}_{U(\mathbb{F}_q)}^U k$ satisfies $B_r/\textup{rad}(B_r) \cong k$
\end{proposition}

\begin{proof}
Indeed, $B_r$ is a mock projective contramodule, as seen in the proof of Theorem \ref{conditions for proper mocks}. By the previous lemma, we know that $k$ is the only simple $k[U]$-contramodule, and furthermore, by the adjunction between induction and restriction we have 
\[\Hom^{U}(B_r,k) \cong \Hom^{U(\mathbb{F}_q)}(k,k) \cong k.\]
It follows that $B_r/\textup{rad}(B_r) \cong k$.
\end{proof}

\subsection{Parabolic subgroups of \texorpdfstring{$G$}{G}} 

\phantom{.}\newline
\noindent
We now turn our attention to contramodules associated to parabolic and Levi subgroups of algebraic groups. Let $G$ be a connected reductive algebraic group, fix a maximal torus $T$ and a Borel subgroup $B$ containing $T$. Then we have maps of coordinate rings $k[G] \longrightarrow k[B] \longrightarrow k[T]$ induced from the inclusions $T \subset B \subset G$. Choose simple roots $\Delta$ such that the root subgroups contained in $B$ correspond to negative roots. Note that our choice of simple roots $\Delta$ determines the set of dominant weights, which we denote $X(T)_+$.

For $J \subset \Delta$, let $P_J$ denote the corresponding parabolic subgroup of $G$ containing $B$, with unipotent radical $U_J$ and Levi factor $L_J$. Let $Z_J := Z(L_J)$ denote the center of the Levi factor, it can be verified that the central characters are given by $X(Z_J) = X(T)/\mathbb{Z}J$. Letting $\pi: X(T) \longrightarrow X(Z_J)$ denote the canonical quotient map, we see that $\pi(\mathbb{Z}\Phi) = \mathbb{Z}I$, where $I = \Delta / J$. It follows immediately that any $k[L_J]$ contramodule $(B,\theta)$ has a central character decomposition of the form
\[B = \prod_{\chi \in \mathbb{Z}I} B_\chi \]
where $B_\chi = \big\{b \in B : \phi(\chi) = \theta(\phi)$ for all $\phi \in \Hom_k(k[L_J],kb) \big\}$. Now, $\Hom_k(k[U_J],k)$ has a natural contramodule structure induced from the $k[L_J]$-comodule structure on $k[U_J]$. The preceding discussion along with \cite[Lemma 3.3.2]{hardesty2017existence} gives the following result.

\begin{lemma}
The $k[L_J]$-contramodule $\Hom_k(k[U_J],k)$ has a central character decomposition
\[\Hom_k(k[U_J],k) = \prod_{\chi \in \mathbb{N}I} \Hom_k(k[U_J]_\chi,k)\]
where $\dim\left(\Hom_k(k[U_J]_\chi,k)\right) < \infty$ for all $\chi \in \mathbb{N}I$.
\end{lemma}

Let $\lambda \in X(T)_+$ be a dominant weight. Denote the simple module of highest weight $\lambda$ by $L(\lambda)$, and let $P(\lambda) \in k[G]$-Contra denote the projective cover of $L(\lambda)$ 

\begin{lemma}
Let $M$ be a finite dimensional right $k[G]$-comodule with linear dual $M^*$, and let $\lambda,\mu \in X_+(T)$. Then:
\[ \dim \Big(\Hom^{k[G]}\big(\Hom(M,P(\lambda)),L(\mu)\big)\Big) = \big[\Hom\big(M^*,L(\mu)\big):L(\lambda)\big].\]
\end{lemma}

\begin{proof}
Let $M$ have basis $\{m_i\}$ and $M^*$ have dual basis $\{m_i^*\}$. Then one checks that we have the following isomorphism:
\begin{align*}
\Hom^{k[G]}\big(\Hom\big(M,P(\lambda)\big),L(\mu)\big) &= \Hom^{k[G]}\big(P(\lambda),\Hom\big(M^*,L(\mu)\big)\big) \\ \Big(f \longmapsto \sum_{i} \big((\phi \circ f)(m_i)\big)(m_i^*)\Big) &\longleftarrow \phi  \\
\psi &\longrightarrow \Big(p \longmapsto \big(\alpha \longmapsto \phi\big(m \longmapsto \alpha(m)p\big)\big)\Big)
\end{align*}
where both $M$ and $M^*$ are viewed as right comodules and $\Hom(-,-)$ is a contramodule via the diagonal action. Since $P(\lambda)$ is projective, $\Hom(P(\lambda),-)$ is exact and so by induction on the composition length, one may show that $\dim \Big(\Hom^{k[G]}\big(\Hom(M,P(\lambda)),L(\mu)\big)\Big)$ is exactly the number of times $L(\lambda)$ appears as a composition factor in $\Hom\big(M^*,L(\mu)\big)$, as required.
\end{proof}

One may inflate a $k[L_J]$ contramodule $M$ to a $k[P_J]$ contramodule via the composition
\[\Hom_k(k[P_J],M) \longrightarrow \Hom_k(k[L_J],M) \longrightarrow M. \]
Given a weight $\lambda \in X(T)$ which is dominant for $L_J$, let $M = L_J(\lambda)$ denote the simple $L_J$ module with highest weight $\lambda$. We denote the inflation by $L_{P_J}(\lambda)$. The projective cover $P_{P_J}(\lambda)$ of $L_{P_J}(\lambda)$ is given by 
\[P_{P_J}(\lambda) = \textup{Ind}_{k[L_J]}^{k[P_J]} P_{L_J}(\lambda) \cong  \Hom_k\big(k[U_J],P_{L_J}(\lambda)\big)\]
where the isomorphism is a consequence of Lemma \ref{semi-direct-induction}.

\begin{lemma}
Let $P$ be a projective $k[P_J]$ contramodule with cofinite dimensional radical, then we have 
\[P|_{k[L_J]} = \prod_{\lambda \in X(T)} P_{L_J}(\lambda)^{n_\lambda}\]
where $n_\lambda < \infty$ for all weights $\lambda \in X(T)$.
\end{lemma}

Much as in the analogous result for injective modules of $P_J$ with finite dimensional socle, \cite[Proposition 3.4.3]{hardesty2017existence}, one must turn to using central characters and the fact the homomorphisms will preserve weight spaces for the proof. All required results to produce an analog of this proof have been proven for contramodules. We leave the necessary modifications to the reader.

Let $F: P_J \longrightarrow P_J$ be the Frobenius morphism and let $(P_J)_rL_J = (F^r)^{-1}(L_J).$ We have the following result. 

\begin{proposition}
Let $B$ be a $k[P_J]$ contramodule with cofinite dimensional radical which is projective as a $k[(P_J)_rL_J]$ contramodule for all $r \geq 1$, then $P$ is projective as a $k[P_J]$ contramodule.
\end{proposition}

\begin{proof}
Since $M/\textup{rad}(M)$ is finite dimensional, it follows that projective cover of $M$ in the category $k[P_J]$-Contra is of the form $\Hom_k(k[U_J],P)$ for some projective $k[L_J]$ contramodule $P$ with cofinite dimensional radical. This gives us a projection 
\[\Hom_k(k[U_J],P) \longrightarrow M.\]
Since by assumption $M|_{k[(P_J)_rL_J]}$ is projective for all $r > 0$ we have projections $\big($of $k[(P_J)_rL_J]$ contramodules$\big)$ of the form $M \longrightarrow \Hom_k\big((k[U_J])_r,P\big)$.
It suffices for us to show that we have $ \displaystyle \Hom_k(k[U_J],P) = \bigcup_r \Hom_k\big((k[U_J])_r,P\big)$, but this follows from the fact that the coordinate ring of an algebraic group is the projective limit of the coordinate rings of its Frobenius kernels. Since colimits commute with colimits, and in particular unions commute with cokernels, we have a projection $M \longrightarrow \Hom_k(k[U_J],P)$ and so $M = \Hom_k(k[U_J],P)$. Thus $M$ is projective as a $k[P_J]$ contramodule, as required.
\end{proof}

We conclude with a corollary in the case when $J = \phi$.  

\begin{corollary}
Let $B$ be a mock projective $k[B]$ contramodule which cofinite dimensional radical, then $B$ is a projective $k[B]$ contramodule.
\end{corollary}

\begin{proof}
Let $J = \phi$. Then $P_\phi = B$ and $L_\phi = T$. The result follows by the previous result and the fact that a contramodule is projective as a $k[B_rT]$ contramodule if and only if it is projective as a $k[B_r]$ contramodule.
\end{proof}

\printbibliography

\end{document}

%% file: macros.tex
\usepackage[margin=1in]{geometry} 

\usepackage{amsmath}
\usepackage{amssymb}
\usepackage{amsthm}
\usepackage{amscd}
\usepackage{mathtools}
\usepackage{graphicx}
\usepackage{enumitem}
\usepackage{faktor}
\usepackage{MnSymbol} 

\usepackage{marginnote}
\usepackage{todonotes}
\usepackage{bbm}
\usepackage{bm}
\usepackage[colorlinks=true,linkcolor=black,citecolor=black,urlcolor=black]{hyperref} 

\usepackage{tikz-cd} 
\tikzcdset{scale cd/.style={every label/.append style={scale=#1}, cells={nodes={scale=#1}}}}
\usetikzlibrary{arrows}

\usepackage[style=alphabetic]{biblatex} 
\renewbibmacro{in:}{} 

\newtheorem{theorem}{Theorem}[section]
\newtheorem*{theorem*}{Theorem}

\newtheorem*{conjecture*}{Conjecture}

\newtheorem*{question*}{Question}

\newtheorem*{guess*}{Guess}
\newtheorem*{Assumption*}{Assumption}

\newtheorem*{problem*}{Problem}
\newtheorem{lemma}[theorem]{Lemma}
\newtheorem*{lemma*}{Lemma}
\newtheorem{proposition}[theorem]{Proposition}
\newtheorem*{proposition*}{Proposition}
\newtheorem{corollary}[theorem]{Corollary}
\newtheorem*{corollary*}{Corollary}

\theoremstyle{definition}

\newtheorem*{definition*}{Definition}

\newtheorem*{remark*}{Remark}

\newtheorem*{example*}{Example}
\newtheorem*{examples*}{Examples}

\newcommand{\End}{\operatorname{End}}
\newcommand{\Hom}{\operatorname{Hom}}
\newcommand{\Ext}{\operatorname{Ext}}

\newcommand{\Image}{\operatorname{Im}}
